\documentclass[12pt]{article}
\usepackage{amsmath,amsthm,amsfonts,amssymb}

\theoremstyle{plain}
\newtheorem{theorem}{Theorem}[section]
\newtheorem{corollary}{Corollary}[section]
\newtheorem{proposition}{Proposition}[section]
\newtheorem{lemma}{Lemma}[section]
\theoremstyle{definition}

\theoremstyle{remark}
\newtheorem{remark}{Remark}[section]

\begin{document}

\title{\large\bf
ON THE TRANSITION DENSITY FUNCTION OF THE DIFFUSION PROCESS 
GENERATED BY THE GRUSHIN OPERATOR
}
\author{\normalsize
Setsuo TANIGUCHI}
\date{}

\maketitle

\begin{abstract}
The short-time asymptotic behavior of the transition density  
function of the diffusion process generated by the
Grushin operator with parameter $\gamma>0$ will be
investigated, by using its explicit expression in terms of
expectation. 
Further the dependence on $\gamma$ of the asymptotics will be
seen.
\end{abstract}

\renewcommand{\thefootnote}{\fnsymbol{footnote}}
\footnote[0]{
\hspace{-9pt}


{\bf Running head}: Short-time asymptotics 
}

\section{Introduction}
Let $d,d^\prime\in \mathbb{N}$ and $\gamma>0$. 
Define the vector fields on $\mathbb{R}^{d+d^\prime}$ by
\[
    V_i=\frac{\partial}{\partial x^i}
    \quad\text{and}\quad
    W_j=|x|^\gamma \frac{\partial}{\partial y^j}
   \quad 
    \text{for $1\le i\le d$ and $1\le j\le d^\prime$},
\]
where $x=(x^1,\dots,x^d)$ and $y=(y^1,\dots,y^{d^\prime})$
are the standard coordinate systems of $\mathbb{R}^d$ and  
$\mathbb{R}^{d^\prime}$, respectively.
The Grushin operator with parameter $\gamma$ is the
differential operator on $\mathbb{R}^{d+d^\prime}$ given by  
\[
    \Delta_{(\gamma)}=\sum_{i=1}^{d} V_i^2 
                 +\sum_{j=1}^{d^\prime} W_j^2.
\]
The operator is also represented as
\[
    \Delta_{(\gamma)}=\Delta_x+|x|^{2\gamma}\Delta_y,
\]
where $\Delta_x$ and $\Delta_y$ are the Laplacians in the
variables $x\in \mathbb{R}^d$ and  
$y\in\mathbb{R}^{d^\prime}$, respectively.
The studies of the Grushin operator go back to
those by Baouendi in 1967 (\cite{baouendi}) and by Grushin
in the beginning of the 1970's (\cite{grushin1,grushin2}). 
After them, many researches corresponding to the operator have
been made.
In particular, in the case when $\gamma$ is an even integer,
the associated heat kernel, i.e. the transition density
function, is studied in details
(cf. \cite{bfi,ccfi,ccggl}).

The main aim of this this paper is to investigate the
short-time asymptotics of the transition density function of
the diffusion process generated by $\frac12\Delta_{(\gamma)}$
for general $\gamma$s in a stochastically analytic method. 
More precisely, let 
$\{\{Z^{(x,y)}(t)\}_{t\in[0,\infty)},
  (x,y)\in\mathbb{R}^d \times \mathbb{R}^{d^\prime}\}$ be
the diffusion process on $\mathbb{R}^{d+d^\prime}$ generated
by $\frac12\Delta_{(\gamma)}$. 
For our purposes, we shall first show the existence of the
transition density function $p_T((x,y),(\xi,\eta))$
of this diffusion process;
\[
    \mathbf{E}[f(Z^{(x,y)}(T))]
    =\int_{\mathbb{R}^d\times \mathbb{R}^{d^\prime}} 
      f(\xi,\eta) p_T((x,y),(\xi,\eta))d\xi d\eta
\]
for any $T>0$, 
$(x,y)\in \mathbb{R}^d\times \mathbb{R}^{d^\prime}$,
and bounded 
$f\in C(\mathbb{R}^d\times \mathbb{R}^{d^\prime})$, where
$\mathbf{E}$ stands for the expectation with respect to the
underlying probability measure.
At the same time, we shall establish an explicit expression
of $p_T((x,y),(\xi,\eta))$ in terms of expectation, by which
we will conclude the continuity of $p_T$. 
See Theorem~\ref{thm.main}.  
Further, with the help of this explicit expression, the
short-time asymptotics of $p_T((x,y),(\xi,\eta))$ will be
investigated in the on-diagonal case
(Theorem~\ref{thm.asympt.on}) and in the off-diagonal case
(Theorems~\ref{thm.asympt.off} and
\ref{thm.asympt.off.00}). 
In all cases, we shall see that the parameter $\gamma$,
i.e.,  the degeneracy of $\Delta_{(\gamma)}$ on the plane
$\{x=0\}$,  affects the asymptotics. 
For example, in Theorem~\ref{thm.asympt.off}, we shall show
the convergence
\[
    \lim_{|\eta-y|\to\infty} |\eta-y|^{-\frac{2}{1+\gamma}}
      \lim_{T\to0}T\log p_T((x,y),(\xi,\eta))
    =C(\gamma),
\]
where $C(\gamma)$ is a constant depending only on
$\gamma$.
The explicit expression of $C(\gamma)$ will be given in the 
theorem.

\section{Preliminaries}\label{sec:preliminaries}

\subsection{Density function}

Let $T>0$.
Denote by $\mathcal{W}_T$ (resp. $\widehat{\mathcal{W}}_T$)
the space of continuous functions $w$ from $[0,T]$ to 
$\mathbb{R}^d$ (resp. $\mathbb{R}^{d^\prime}$) with
$w(0)=0$, and by $\mu_T$ (resp. $\widehat{\mu}_T$) the Wiener
measure on $\mathcal{W}_T$ (resp. $\widehat{\mathcal{W}}_T$).
The product space $\mathcal{W}_T\times\widehat{\mathcal{W}}_T$
equipped with the product measure 
$\mu_T\times\widehat{\mu}_T$ is the $(d+d^\prime)$-dimensional
Wiener space.
Let $b=\{b(t)=(b^1(t),\dots,b^d(t))\}_{t\in[0,T]}$
(resp. $\widehat{b}=\{\widehat{b}(t)
     =(\widehat{b}^1(t),\dots,\widehat{b}^{d^\prime}(t))
      \}_{t\in[0,T]}$)
be the coordinate process on $\mathcal{W}_T$ 
(resp. $\widehat{\mathcal{W}}_T$);
$b(t):\mathcal{W}_T\ni w\mapsto b(t)(w)=w(t)\in \mathbb{R}^d$
and 
$\widehat{b}(t):\widehat{\mathcal{W}}_T\ni \widehat{w}
  \mapsto 
  \widehat{b}(t)(\widehat{w})=\widehat{w}(t)
  \in \mathbb{R}^{d^\prime}$.
Denote by $\mathcal{F}_t$ the $\sigma$-field on
$\mathcal{W}_T$ generated by 
$\{b(s)^{-1}(A)\mid s\le t, 
 A\in \mathcal{B}(\mathbb{R}^d)\}$, 
$\mathcal{B}(\mathbb{R}^d)$ being the Borel field of
$\mathbb{R}^d$. 
In the standard manner, 
$\mathcal{F}_t$ can be thought of as a $\sigma$-field on
 $\mathcal{W}_T\times\widehat{\mathcal{W}}_T$.

Let $N\in \mathbb{N}$, and 
$\phi=\{\phi(t)
  =(\phi_j^i(t))_{1\le i\le N,1\le j\le d^\prime}\}_{t\in[0,T]}$ 
be an $\mathbb{R}^{N\times d^\prime}$-valued
($\mathcal{F}_t$)-progressively measurable process on 
$\mathcal{W}_T$ satisfying that
\begin{equation}\label{eq:phi.int}
    \int_0^T \|\phi(t)\|^2 dt\in L^{\infty-}(\mu_T)
    \equiv \bigcap_{p\in(1,\infty)}L^p(\mu_T),
\end{equation}
where $\mathbb{R}^{N\times d^\prime}$ is the space of real 
$N\times d^\prime$-matrices and $\|\cdot\|$ denotes the
Hilbert-Schmidt norm on it.
Let $\mathcal{G}_t$ be the $\sigma$-field on
$\mathcal{W}_T\times \widehat{\mathcal{W}}_T$ generated by 
$\bigl\{(b(s),\widehat{b}(s))^{-1}(A) \mid
   s\le t, A\in \mathcal{B}(\mathbb{R}^{d+d^\prime})\bigr\}$.
Thinking of $\phi$ as a ($\mathcal{G}_t$)-progressively
measurable stochastic process on 
$\mathcal{W}_T\times \widehat{\mathcal{W}}_T$, 
we define the $\mathbb{R}^N$-valued random variable $F$ by
the stochastic integral 
\[
    F=\int_0^T \phi(t) d\widehat{b}(t),
\]
that is, the $i$th component $F^i$ of $F$ is given by  
\[
    F^i=\sum_{j=1}^{d^\prime} \int_0^T \phi_j^i(t)
        d\widehat{b}^j(t) \quad\text{for }1\le i\le N.
\]
Set 
\[
   V_\phi=\int_0^T \phi(t)\phi(t)^\dagger dt,
\]
where $A^\dagger$ stands for the transposed matrix of $A$.

\begin{lemma}\label{lem.density.phi}
Suppose that
\begin{equation}\label{l.density.phi.1}
   \frac1{\det V_\phi}   \in L^{\infty-}(\mu_T).
\end{equation}
Define $q_F:\mathbb{R}^N\to \mathbb{R}$ by
\[
    q_F(\eta)
    =\int_{\mathcal{W}_T}
       \frac1{\sqrt{2\pi}^N\sqrt{\det V_\phi}}
       \exp\biggl(-\frac{\langle V_\phi^{-1}\eta,\eta
          \rangle_{\mathbb{R}^N}}2\biggr) 
       d\mu_T 
    \quad \text{for }\eta\in \mathbb{R}^N,
\]
where $\langle\cdot,\cdot\rangle_{\mathbb{R}^N}$ stands for
the inner product of $\mathbb{R}^N$.
Then $q_F$ is the $C^\infty$-distribution density function
of $F$ with respect to the Lebesgue measure on
$\mathbb{R}^N$. 
\end{lemma}

\begin{proof}
Let $\widetilde{V_\phi}$ be the cofactor matrix of $V_\phi$.
Then, $V_\phi^{-1}=(\det V_\phi)^{-1}\widetilde{V_\phi}$.
By the assumptions \eqref{eq:phi.int} and
\eqref{l.density.phi.1}, 
$\|V_\phi^{-1}\|\in L^{\infty-}(\mu_T\times\widehat{\mu}_T)$. 
Applying the dominated convergence theorem, we see that
$q_F\in C^\infty(\mathbb{R}^N)$.

For a Borel measurable 
$A:[0,T]\to \mathbb{R}^{N\times d^\prime}$, set
\[
    \mathcal{I}_A=\int_0^T A(t)d\widehat{b}(t)
    \quad \text{and}\quad
    v_A=\int_0^T A(t)A(t)^\dagger dt.
\]
It holds that
\[
    \mathbf{E}[f(F)|\mathcal{F}_T]
    =\mathbf{E}[f(\mathcal{I}_A)]\Bigr|_{A=\phi} 
    \quad\text{for any bounded }f\in C(\mathbb{R}^N),
\]
where $\mathbf{E}[\,\cdot\,|\mathcal{F}_T]$ denotes the
conditional expectation with respect to
$\mu_T\times\widehat{\mu}_T$ given $\mathcal{F}_T$
and $\mathbf{E}[\,\cdot\,]$ does the expectation with
respect to $\mu_T\times\widehat{\mu}_T$.
Moreover, if $\det v_A\ne0$, then $\mathcal{I}_A$ obeys the 
$N$-dimensional normal distribution with mean $0$ and
covariance $v_A$. 
Thus, it holds
\begin{align*}
    & \int_{\mathcal{W}_T\times\widehat{\mathcal{W}}_T} f(F)
      d(\mu_T\times\widehat{\mu}_T)
      =\int_{\mathcal{W}_T\times\widehat{\mathcal{W}}_T} 
       \mathbf{E}[f(F)|\mathcal{F}_T]  
       d(\mu_T\times\widehat{\mu}_T)
    \\
    & \qquad
      =\int_{\mathcal{W}_T\times\widehat{\mathcal{W}}_T} 
       \biggl(\int_{\mathbb{R}^N} 
       \frac{f(\eta)}{\sqrt{2\pi}^N\sqrt{\det V_\phi}}
       \exp\biggl(-\frac{\langle V_\phi^{-1}\eta,\eta
          \rangle_{\mathbb{R}^N}}2\biggr) d\eta\biggr)
          d(\mu_T\times\widehat{\mu}_T)
\end{align*}
for any bounded $f\in C(\mathbb{R}^N)$.
This implies
\[
    \int_{\mathcal{W}_T\times\widehat{\mathcal{W}}_T} f(F)
      d(\mu_T\times\widehat{\mu}_T)
    =\int_{\mathbb{R}^N} f(\eta)q_F(\eta)d\eta,
\]
which means $q_F$ is the distribution density function of
$F$ with respect to the Lebesgue measure.
\end{proof}

\subsection{Exponential integrability}

Let $\mathcal{W}_T^0=\{w\in \mathcal{W}_T\mid w(T)=0\}$,
$\mu_T^0$ be the pinned Wiener measure on it, and
$\beta=\{\beta(t)=(\beta^1(t),\dots,\beta^d(t))\}_{t\in[0,T]}$
be the coordinate process on $\mathcal{W}_T^0$, i.e.,
$\beta(t)=b(t)|_{\mathcal{W}_T^0}$.
$\mathcal{W}_T^0$ is a real separable Banach space with the
uniform convergence norm $\|\,\cdot\,\|_{\mathcal{W}_T^0}$
inherited from $\mathcal{W}_T$; 
\[
    \|w\|_{\mathcal{W}_T^0}
    =\sup_{t\in[0,T]}|w(t)|
    \quad\text{for }w\in \mathcal{W}_T^0.
\]

Given $\theta\in(0,\frac12)$ and $p\in(1,\infty)$ with
$p\theta>1$, define
\[
    \|\psi\|_{T,p,\theta}
    =\biggl(\int_{(0,T)^2} 
         \frac{|\psi(u)-\psi(v)|^p}{|u-v|^{1+p\theta}}
       dudv\biggr)^{\frac1{p}}
    \quad\text{for }\psi\in C([0,T];\mathbb{R}^d).
\]
The aim of this subsection is to show

\begin{proposition}\label{prop.exp.int}
There exists a $\delta\in(0,\infty)$ such that
$\exp(\delta \|\cdot\|_{T,p,\theta}^2)\in L^1(\mu_T^0)$.
\end{proposition}

For $x,\xi\in \mathbb{R}^d$, define 
$\ell^{T,x,\xi}\in C([0,T];\mathbb{R}^d)$ by
\[
    \ell^{T,x,\xi}(t)=x+\frac{t}{T}(\xi-x)
    \quad\text{for }t\in[0,T].
\]
This proposition yields the following estimation.

\begin{corollary}\label{cor.exp.int}
For the same $\delta$ as described in
Proposition~\ref{prop.exp.int}, it holds
\[
    \sup_{|\xi-x|\le R} \int_{\mathcal{W}_T^0} 
      \exp\biggl(\frac{\delta}2
         \|\beta+\ell^{T,x,\xi}\|_{T,p,\theta}^2\biggr)
      d\mu_T^0 <\infty
     \quad\text{for any }R\in[0,\infty).
\]
\end{corollary}

\begin{proof}
By a straightforward computation, we have
\[
    \|\ell^{T,x,\xi}\|_{T,p,\theta}
    =\biggl(\frac{2}{T^{p\theta-1}
             (1-\theta)p\{(1-\theta)p+1\}}
     \biggr)^{\frac1{p}}|\xi-x|.
\]
Since 
$\|\beta+\ell^{T,x,\xi}\|_{T,p,\theta}\le
 \|\beta\|_{T,p,\theta}+\|\ell^{T,x,\xi}\|_{T,p,\theta}$,
the desired estimation follows from
Proposition~\ref{prop.exp.int}. 
\end{proof}

In the remaining of this subsection, we prove
Proposition~\ref{prop.exp.int}.
Let $\mathcal{H}_T^0$ be the Cameron-Martin subspace of
$\mathcal{W}_T^0$; it is the space of all 
$h\in \mathcal{W}_T^0$ which is absolutely continuous and
possesses a square integrable derivative $h^\prime$ on
$[0,T]$. 
$\mathcal{H}_T^0$ is a real separable Hilbert space with the
norm $\|\cdot\|_{\mathcal{H}_T^0}$ corresponding to the inner
product
\[
    \langle h,g\rangle_{\mathcal{H}_T^0}
    =\int_0^T \langle h^\prime(t),g^\prime(t)
        \rangle_{\mathbb{R}^d} dt
    \quad\text{for }h,g\in \mathcal{H}_T^0.
\]

Observe 
\begin{equation}\label{eq:h-conti}
    |h(t)-h(s)|\le \|h\|_{\mathcal{H}_T^0}|t-s|^{\frac12}
    \quad\text{for $h\in \mathcal{H}_T^0$ and }
     s,t\in[0,T].
\end{equation}
Hence
\begin{equation}\label{p.exp.int.20}
    \|h\|_{T,p,\theta} \le 
      \biggl(\frac{2T^{(\frac12-\theta)p+1}}{
          (\frac12-\theta)p
          \bigl((\frac12-\theta)p+1\bigr)}
      \biggr)^{\frac1{p}}
      \|h\|_{\mathcal{H}_T^0}
    \quad\text{for }h\in \mathcal{H}_T^0.
\end{equation}
Let $\mathcal{W}_{T,p,\theta}^0$ be the completion of 
$\mathcal{H}_T^0$ with respect to $\|\cdot\|_{T,p,\theta}$.

The Garsia-Rodemich-Rusey lemma 
(cf. \cite[Lemma~3.1]{st-grushin}) asserts  
\[
    |\psi(t)-\psi(s)|
    \le 2^{3+\frac2{r}}\frac{r}{\alpha-2}
       \biggl(\int_{(0,T)^2}
       \frac{|\psi(u)-\psi(v)|^r}{|u-v|^\alpha}
        dudv\biggr)^\frac1{r}
       |t-s|^{\frac{\alpha-2}{r}}
\]
for any $\alpha>2$, $r>0$, $s,t\in[0,T]$, and 
$\psi\in C([0,T];\mathbb{R}^d)$ with 
$\int_{(0,T)^2} \frac{|\psi(u)-\psi(v)|^r}{|u-v|^\alpha}
        dudv<\infty$.
Setting $r=p$ and $\alpha=1+p\theta$, we obtain 
\begin{equation}\label{eq:garsia0}
    |\psi(t)-\psi(s)|
    \le 2^{3+\frac2{p}}\frac{p}{p\theta-1}
    \|\psi\|_{T,p,\theta}|t-s|^{\theta-\frac1{p}}
\end{equation}
for $s,t\in[0,T]$ and $\psi\in C([0,T];\mathbb{R}^d)$ with
$\|\psi\|_{T,p,\theta}<\infty$.
This inequality with $s=0$ yields 
\begin{equation}\label{eq:garsia1}
    \sup_{t\in[0,T]}|\psi(t)|
    \le 2^{3+\frac2{p}}\frac{p}{p\theta-1}
        T^{\theta-\frac1{p}} \|\psi\|_{T,p,\theta}
    \end{equation}
for $\psi\in \mathcal{W}_T^0$ with 
$\|\psi\|_{T,p,\theta}<\infty$.

Let $\{h_n\}_{n=1}^\infty\subset \mathcal{H}_T^0$ be a Cauchy
sequence with respect to $\|\cdot\|_{T,p,\theta}$.
By \eqref{eq:garsia1}, it holds
\[
    \|h_n-h_m\|_{\mathcal{W}_T^0}
    \le 2^{3+\frac2{p}}\frac{p}{p\theta-1}
        T^{\theta-\frac1{p}} \|h_n-h_m\|_{T,p,\theta}
    \quad\text{for }n,m\in \mathbb{N}.
\]
Hence $\{h_n\}_{n=1}^\infty$ converges in $\mathcal{W}_T^0$ to
some point in $\mathcal{W}_T^0$.
Moreover, if two Cauchy sequences
$\{h_n\}_{n=1}^\infty,\{\hat{h}_n\}_{n=1}^\infty
  \subset \mathcal{H}_T^0$ 
with respect to $\|\cdot\|_{T,p,\theta}$ are equivalent,
i.e.,
$\lim_{n\to\infty}\|h_n-\hat{h}_n\|_{T,p,\theta}=0$, then,
by \eqref{eq:garsia1} again,
$\lim_{n\to\infty}\|h_n-\hat{h}_n\|_{\mathcal{W}_T^0}=0$.
Thus, each equivalent class of Cauchy sequences 
with respect to $\|\cdot\|_{T,p,\theta}$ is 
identified with the limit point in $\mathcal{W}_T^0$.
In this manner, we obtain the inclusion 
\begin{equation}\label{eq:inclusion}
    \mathcal{W}_{T,p,\theta}^0\subset \mathcal{W}_T^0.
\end{equation}
Further, \eqref{eq:garsia1} also yields the continuity of
this inclusion;
\begin{equation}\label{eq:continuous}
    \|w\|_{\mathcal{W}_T^0}
    \le 2^{3+\frac2{p}}\frac{p}{p\theta-1}
        T^{\theta-\frac1{p}} \|w\|_{T,p,\theta}
    \quad\text{for }w\in \mathcal{W}_{T,p,\theta}^0.
\end{equation}
Denoting by $\mathcal{W}_T^{0*}$ and
$\mathcal{W}_{T,p,\theta}^{0*}$ the duals spaces of
$\mathcal{W}_T^0$ and $\mathcal{W}_{T,p,\theta}^0$,
respectively, we then have
\[
    \mathcal{W}_T^{0*}\subset \mathcal{W}_{T,p,\theta}^{0*}.
\]

We moreover prepare a functional analytical lemma.

\begin{lemma}\label{lem.C^alpha}
{\rm (i)} 
Let $\alpha\in(\theta,\frac12)$ and denote by 
$C_T^\alpha\subset \mathcal{W}_T^0$ be the space of all
$\alpha$-H\"older continuous $w\in \mathcal{W}_T^0$.
Then $C_T^\alpha\subset \mathcal{W}_{T,p,\theta}^0$.
\\
{\rm (ii)}
$\mathcal{W}_T^{0*}$ is dense in
$\mathcal{W}_{T,p,\theta}^{0*}$.
\end{lemma}

\begin{proof}
(i) 
Let $f\in \mathcal{C}_T^\alpha$.
Extend $f$ to $\overline{f}\in C(\mathbb{R})$ by setting
$\overline{f}=0$ outside of $[0,T]$.
Defining
\[
    C_f=\sup_{s,t\in[0,T],s\ne t} 
        \frac{|f(t)-f(s)|}{|t-s|^\alpha}<\infty,
\]
we have
\begin{equation}\label{l.C^alpha.21}
    |\overline{f}(t)-\overline{f}(s)|
    \le C_f|t-s|^\alpha \quad
    \text{for any }s,t\in \mathbb{R}.
\end{equation}

Take a non-negative $\varphi\in C_0^\infty(\mathbb{R})$ such
that $\varphi=0$ on $(-\infty,0)$,
and $\int_{\mathbb{R}} \varphi(u)du=1$.
Set $\varphi_n(y)=n\varphi(ny)$ for $y\in \mathbb{R}$.
Put
\[
    f_n(t)=\int_{\mathbb{R}} \overline{f}(t-u)\varphi_n(u)du
      -\frac{t}{T} \int_{\mathbb{R}}\overline{f}(T-u)
           \varphi_n(u)du
    \quad\text{for }t\in[0,T].
\]
Then $f_n(0)=f_n(T)=0$ and 
$f_n\in C^\infty([0,T];\mathbb{R}^d)$.
In particular, $f_n\in \mathcal{H}_T^0$.

Since $\int_{\mathbb{R}}\varphi_n(u)du=1$ and
$\overline{f}(T)=0$, by \eqref{l.C^alpha.21},  
\begin{align*}
    |f_n(t)-f(t)|
    & =\biggl|\int_{\mathbb{R}}
       (\overline{f}(t-u)-\overline{f}(t))
                    \varphi_n(u)du
     -\frac{t}{T}\int_{\mathbb{R}} 
          (\overline{f}(T-u)-\overline{f}(T))
                        \varphi_n(u)du\biggr|
    \\
    & \le 
      2C_f n^{-\alpha} \int_{\mathbb{R}} |v|^\alpha
      \varphi(v)dv
     \quad\text{for }t\in[0,T].
\end{align*}
Thus 
\begin{equation}\label{l.C^alpha.22}
    \|f_n-f\|_{\mathcal{W}_T^0}\to 0
    \quad \text{as }n\to\infty.
\end{equation}

By \eqref{l.C^alpha.21}, 
$\sup_{t\in \mathbb{R}}|\overline{f}(t)|\le C_fT^\alpha$.
Due to this and \eqref{l.C^alpha.21} again, we have
\begin{align*}
    |f_n(t)-f_n(s)|
    & \le \int_{\mathbb{R}} |\overline{f}(t-u)
             -\overline{f}(s-u)| \varphi_n(u)du
        +\frac{|t-s|}{T} \int_{\mathbb{R}} 
            |\overline{f}(T-u)|\varphi_n(u)du
    \\
    & \le 2C_f|t-s|^\alpha
    \quad\text{for }t,s\in[0,T].
\end{align*}
This yields the domination
\[
    \frac{|(f_n(t)-f_m(t))-(f_n(s)-f_m(s))|^p}{
         |t-s|^{1+p\theta}}
    \le (4C_f)^p|t-s|^{(\alpha-\theta)p-1}
    \quad\text{for }t,s\in[0,T].
\]
By \eqref{l.C^alpha.22} and the dominated convergence
theorem, we obtain 
\[
    \lim_{n,m\to\infty}\|f_n-f_m\|_{T,p,\theta}=0.
\]
Thus $\{f_n\}_{n=1}^\infty\subset \mathcal{H}_T^0$ is a Cauchy
sequence with respect to $\|\cdot\|_{\mathcal{W}_{T,p,\theta}}$. 
Due to \eqref{l.C^alpha.22} and the inclusion 
\eqref{eq:inclusion}, 
we have $f\in \mathcal{W}_{T,p,\theta}^0$.
\\
(ii)
For $-\infty\le a<b\le\infty$, 
let $L^p(a,b)$ be the $L^p$-space with respect to the
Lebesgue measure on $(a,b)$.
Put
\[
    W_\theta^p(a,b)=\{f\in L^p(a,b)\mid
      \|f\|_{W_\theta^p(a,b)}<\infty\},
\]
where 
\[
    \|f\|_{W_\theta^p(a,b)}
    =\biggl(\int_{(a,b)}|f(u)|^p du\biggr)^{\frac1{p}}
     +\biggl(\int_{(a,b)^2}
         \frac{|f(u)-f(v)|^p}{|u-v|^{1+\theta p}} dudv
        \biggr)^{\frac1{p}}.
\]

Let $-\infty<a<b<\infty$ and 
$R_{a,b}:L^p(-\infty,\infty)\to L^p(a,b)$ be the restriction
mapping onto $(a,b)$. 
It is known (\cite{triebel}) that there exists a bounded
linear operator 
$S_{a,b}:W_\theta^p(a,b)\to W_\theta^p(-\infty,\infty)$
such that the composition $R_{a,b}\circ S_{a,b}$ is
the identity mapping of $W_\theta^p(a,b)$. 
Thus $W_\theta^p(a,b)$ can be thought of as a closed
subspace of $W_\theta^p(-\infty,\infty)$.
Since $W_\theta^p(-\infty,\infty)$ is reflexive
(\cite{triebel}), $W_\theta^p(a,b)$ is also reflexive.

By \eqref{eq:continuous}, $\|\cdot\|_{W_\theta^p(0,T)}$ and
$\|\cdot\|_{T,p,\theta}$ are equivalent on 
$\mathcal{W}_{T,p,\theta}^0$.
Hence $\mathcal{W}_{T,p,\theta}^0$ is a closed subspace of
$W_\theta^p(0,T)$, and hence is reflexive.

Since $\mathcal{W}_{T,p,\theta}^0$ is reflexive and imbedded
continuously in $\mathcal{W}_T^0$ by \eqref{eq:continuous},
it is an elementary exercise of functional analysis to show
$\mathcal{W}_T^{0*}$ is dense in
$\mathcal{W}_{T,p,\theta}^{0*}$.
\end{proof}

As is well known, $\mu_T^0(C_T^\alpha)=1$.
Hence, by the above lemma, we obtain the probability measure
$\mu_{T,p,\theta}^0$ on $\mathcal{W}_{T,p,\theta}$ by restricting
$\mu_T^0$ to $\mathcal{W}_{T,p,\theta}^0$. 

\begin{proof}[Proof of Proposition~\ref{prop.exp.int}]
The proof completes once we have shown
$(\mathcal{W}_{T,p,\theta}^0,\mathcal{H}_T^0,\mu_{T,p,\theta}^0)$
is an abstract Wiener space;
(i) $\mathcal{H}_T^0$ is imbedded in
$\mathcal{W}_{T,p,\theta}^0$ densely and continuously, and 
(ii)~it holds
\begin{equation}\label{p.exp.int.21}
    \int_{\mathcal{W}_{T,p,\theta}^0} e^{\sqrt{-1}\ell}
    d\mu_{T,p,\theta}^0 
    =\exp\biggl(-\frac{\|\ell\|_{\mathcal{H}_T^0}^2}2\biggr)
    \quad\text{for any }\ell\in \mathcal{W}_{T,p,\theta}^{0*},
\end{equation}
where we have identified $\mathcal{H}_T^0$ with its duals
space and thought of $\mathcal{W}_{T,p,\theta}^{0*}$ as a
subspace of $\mathcal{H}_T^0$. 
In fact, the Fernique theorem (\cite{ledoux}) applied to
this abstract Wiener space yields the existence of
$\delta\in(0,\infty)$ such that
\[
    \int_{\mathcal{W}_{T,p,\theta}^0} \exp(\delta
     \|\cdot\|_{T,p,\theta}^2) d\mu_{T,p,\theta}^0<\infty.
\]
By the definition of $\mu_{T,p,\theta}^0$, this means 
$\exp(\delta\|\cdot\|_{T,p,\theta}^2)
 \in L^1(\mu_T^0)$.

The denseness and the continuity of the imbedding of
$\mathcal{H}_T^0$ into $\mathcal{W}_{T,p,\theta}^0$ follow
immediately from the definition of
$\mathcal{W}_{T,p,\theta}^0$ and \eqref{p.exp.int.20}.
To complete the proof, it remains to show
\eqref{p.exp.int.21}.
To do this, given $\ell\in \mathcal{W}_{T,p,\theta}^{0*}$,
apply Lemma~\ref{lem.C^alpha}(ii) to take a sequence 
$\{\ell_n\}_{n=1}^\infty \subset \mathcal{W}_T^{0*}$
converging to $\ell$ in $\mathcal{W}_{T,p,\theta}^{0*}$.
Since $(\mathcal{W}_T^0,\mathcal{H}_T^0,\mu_T^0)$ is an
abstract Wiener space, we have 
\[
    \int_{\mathcal{W}_{T,p,\theta}^0} e^{\sqrt{-1}\ell_n}
    d\mu_{T,p,\theta}^0
    =\int_{\mathcal{W}_T^0} e^{\sqrt{-1}\ell_n} d\mu_T^0
    =\exp\biggl(-\frac{\|\ell_n\|_{\mathcal{H}_T^0}^2}2
                \biggr)
    \quad\text{for }n\in \mathbb{N}.
\]
By \eqref{p.exp.int.20}, 
$\mathcal{W}_{T,p,\theta}^{0*}$ is imbedded in
$\mathcal{H}_T^0$ continuously.
Then, letting $n\to\infty$ in the above identity, we arrive
at \eqref{p.exp.int.21}.
\end{proof}

\section{Transition density function and on-diagonal
  asym\-ptotics}\label{sec:on-diagonal}

We continue to use the same notation introduced in the
preceding sections.
Our first aim of this section is to give an explicit
expression of the transition density function
$p_T((x,y),(\xi,\eta))$ as follows.

\begin{theorem}\label{thm.main}
Let $x,\xi\in \mathbb{R}^d$.
Define the random variable $v_{T,x,\xi}$ on
$\mathcal{W}_T^0$ by 
\[
    v_{T,x,\xi}=\int_0^T
    |(\beta+\ell^{T,x,\xi})(t)|^{2\gamma} dt.
\]
{\rm(i)}
$v_{T,x,\xi}^{-1}\in L^{\infty-}(\mu_T^0)$, and it holds
\begin{equation}\label{t.main.1}
    p_T((x,y),(\xi,\eta))
    = \frac1{\sqrt{2\pi}^{d+d^\prime}}
      \frac1{\sqrt{T}^d} 
      \exp\Bigl(-\frac{|\xi-x|^2}{2T}\Bigr)
      \int_{\mathcal{W}_T^0}  
        v_{T,x,\xi}^{-\frac{d^\prime}2} 
        \exp\biggl(-\frac{|\eta-y|^2}{2 v_{T,x,\xi}}
          \biggr) d\mu_T^0.
\end{equation}
{\rm (ii)}
$p_T$ is continuous on 
$(\mathbb{R}^d\times \mathbb{R}^{d^\prime})^2$.

\end{theorem}

To prove the theorem, we prepare several lemmas.

\begin{lemma}\label{lem.max.prob}
For $\zeta\in \mathbb{R}^{d}$, put
$M_{T,\zeta}=\max_{t\in[0,T]} |(\beta+\ell^{T,0,\zeta})(t)|$.
Then it holds
\begin{equation}\label{l.max.prob.1}
    \mu_T^0(M_{T,\zeta}\le a)
    \le \exp\Bigl(\frac{|\zeta|^2}{2T}\Bigr)
         \frac{\sqrt{2\pi T}^d}{a^d}
        \frac{\exp\Bigl(-\dfrac{d\pi^2 T}{8a^2}\Bigr)}{
          \Bigl(1-\exp\Bigl(-\dfrac{\pi^2 T}{8a^2}\Bigr)
        \Bigr)^d}
   \quad\text{for any }a>0.
\end{equation}
In particular,
\begin{equation}\label{l.max.prob.2}
    \sup_{|\zeta|\le R} \|M_{T,\zeta}^{-1}\|_{L^p(\mu_T^0)}
    <\infty
    \quad\text{for any $R>0$ and }p\in(1,\infty).
\end{equation}
\end{lemma}

\begin{proof}
Let $(\beta+\ell^{T,0,\zeta})^i(t)$ be the $i$th
component of $(\beta+\ell^{T,0,\zeta})(t)$.
It was shown in \cite[p.429]{st-grushin} that
\begin{equation}\label{l.max.prob.21}
    \mu_T^0\Bigl(\max_{t\in[0,T]} 
      |(\beta+\ell^{T,0,\zeta})^i(t)|\le a\Bigr)
    \le \exp\Bigl(\frac{(\zeta^i)^2}{2T}\Bigr)
        \frac{\sqrt{2\pi T}}{a}
        \frac{\exp\Bigl(-\dfrac{\pi^2 T}{8a^2}\Bigr)}{
          1-\exp\Bigl(-\dfrac{\pi^2 T}{8a^2}\Bigr)},
\end{equation}
where $\zeta=(\zeta^1,\dots,\zeta^d)$.
The independence of components of 
$(\beta+\ell^{T,0,\zeta})(t)$ implies
\[
    \mu_T^0(M_{T,\zeta}\le a)
    \le \prod_{i=1}^d
    \mu_T^0\Bigl(\max_{t\in[0,T]} 
     |(\beta+\ell^{T,0,\zeta})^i(t)|\le a\Bigr).
\]
Plugging \eqref{l.max.prob.21} into this, we obtain
\eqref{l.max.prob.1}.
\end{proof}

For $x,\xi\in \mathbb{R}^d$, 
define the random variable $B_{T,x,\xi}$ on
$\mathcal{W}_T^0$ by 
\[
    B_{T,x,\xi}=\|\beta+\ell^{T,x,\xi}\|_{T,12,\frac14}.
\]
It holds $B_{T,x,\xi}=B_{T,0,\xi-x}$.
By Corollary~\ref{cor.exp.int}, there exists
$\delta\in(0,\infty)$ such that
\begin{equation}\label{eq:exp.int}
    \sup_{|\xi-x|\le R}
    \int_{\mathcal{W}_T^0} \biggl\{\exp(\delta B_{T,0,0}^2)
          +\exp\biggl(\frac{\delta}2 B_{T,x,\xi}^2\biggr)  
           \biggr\}d\mu_T^0<\infty
    \quad \text{for any }R>0.
\end{equation}
In particular, $B_{T,x,\xi}<\infty$ $\mu_T^0$-a.s.
Moreover, by virtue of \eqref{eq:garsia0}, 
\begin{equation}\label{eq:garsia}
    |(\beta+\ell^{T,x,\xi})(t)
          -(\beta+\ell^{T,x,\xi})(s)|
    \le 96 B_{T,0,\xi-x}|t-s|^{\frac16}
    \quad \text{for }t,s\in[0,T],
    \quad\text{$\mu_T^0$-a.s.}
\end{equation}

\begin{lemma}\label{lem.non-deg}
It holds
\begin{equation}\label{l.non-deg.1}
    \sup_{|x|,|\xi|\le R}
      \|v_{T,x,\xi}^{-1}\|_{L^p(\mu_T^0)}<\infty
    \quad\text{for any $R>0$ and }p\in(1,\infty).
\end{equation}
In particular,
$v_{T,x,\xi}^{-1}\in L^{\infty-}(\mu_T^0)$ for any 
$x,\xi\in \mathbb{R}^d$.
\end{lemma}

\begin{proof}
Set
$A_{T,x,\xi}=\{M_{T,\xi-x}\ge 4|x|\}$ and 
take a random variable $\sigma\in[0,T]$ such that 
$M_{T,\xi-x}=|(\beta+\ell^{T,0,\xi-x})(\sigma)|$.

On $A_{T,x,\xi}$,
by \eqref{eq:garsia} with $s=\sigma$, 
if $t\in[0,T]$ satisfies
$96B_{T,0,\xi-x}|\sigma-t|^{\frac16}\le\frac{M_{T,\xi-x}}2$, then
\[
    |(\beta+\ell^{T,x,\xi})(t)|
    \ge \frac{M_{T,\xi-x}}2-|x|
    \ge \frac{M_{T,\xi-x}}4.
\]
This implies
\begin{equation}\label{l.non-deg.30}
    v_{T,x,\xi}\ge \Bigl(\frac{M_{T,\xi-x}}4\Bigr)^{2\gamma}
      \biggl( \frac{M_{T,\xi-x}^6}{192^6 B_{T,0,\xi-x}^6}
      \wedge\frac{T}2\biggr)
    \quad\text{on }A_{T,x,\xi}.
\end{equation}
Hence we have
\[
    v_{T,x,\xi}^{-1} \boldsymbol{1}_{A_{T,x,\xi}}
    \le 4^{2\gamma} M_{T,\xi-x}^{-2\gamma}
      \biggl( \frac{192^6 B_{T,0,\xi-x}^6}{M_{T,\xi-x}^6}
      +\frac{2}{T}\biggr)\boldsymbol{1}_{A_{T,x,\xi}}.
\]
By \eqref{l.max.prob.2} and \eqref{eq:exp.int}, we obtain
\begin{equation}\label{l.non-deg.31}
    \sup_{|\xi-x|\le R} \int_{A_{T,x,\xi}} v_{T,x,\xi}^{-p}
     d\mu_T^0 <\infty
     \quad\text{for any $R>0$ and }p\in(1,\infty).
\end{equation}

Let $A_{T,x,\xi}^c$ be the complement set of $A_{T,x,\xi}$.
If $x=0$, then $A_{T,x,\xi}^c=\emptyset$, and hence
\[
    \int_{A_{T,x,\xi}^c} v_{T,x,\xi}^{-p} d\mu_T^0 =0.
\]
Suppose $x\ne0$. 
By \eqref{eq:garsia} with $s=0$, 
if $t\in[0,T]$ satisfies 
$96 B_{T,0,\xi-x}t^{\frac16}\le \frac{|x|}2$, then
\[
    |(\beta+\ell^{T,x,\xi})(t)|
    \ge \frac{|x|}2.
\]
Thus
\begin{equation}\label{l.non-deg.32}
    v_{T,x,\xi}\ge \Bigl(\frac{|x|}2\Bigr)^{2\gamma}
    \biggl(\frac{|x|^6}{192^6 B_{T,0,\xi-x}^6}\wedge T\biggr).
    \end{equation}
Hence, by \eqref{l.max.prob.1} with $a=4|x|$, we have
\begin{align*}
    \int_{A_{T,x,\xi}^c} v_{T,x,\xi}^{-p} d\mu_T^0
    \le & \frac{2^{2\gamma p}}{|x|^{(2\gamma+6) p}}
     \biggl\| 192^6 B_{T,0,\xi-x}^6+\frac{|x|^6}{T}
        \biggr\|_{L^{2p}(\mu_T^0)}^p
    \\
    & \qquad
      \times
      \exp\Bigl(\frac{|\xi-x|^2}{4T}\Bigr)
      \frac{\sqrt{2\pi T}^{\frac{d}2}}{2^d|x|^{\frac{d}2}}
      \frac{\exp\biggl(-\dfrac{d\pi^2T}{256|x|^2}\biggr)}{
         \biggl(1-\exp\biggl(-\dfrac{\pi^2}{128|x|^2}
             \biggr)\biggr)^{\frac{d}2}}.
\end{align*}
We now have
\[
    \sup_{|x|,|\xi|\le R} \int_{A_{T,x,\xi}^c}
        v_{T,x,\xi}^{-p} d\mu_T^0 <\infty
    \quad\text{for any $R>0$ and }p\in(1,\infty).
\]
In conjunction with \eqref{l.non-deg.31}, this implies
\eqref{l.non-deg.1}.
\end{proof}

\begin{lemma}\label{lem.density.Y}
Let $F_{T,x}=\int_0^T |x+b(t)|^\gamma d\widehat{b}(t)$ and 
define the function $q_{T,x,\xi}$ by
\[
    q_{T,x,\xi}(\eta)
    =\int_{\mathcal{W}_T^0}
     \frac1{\sqrt{2\pi v_{T,x,\xi}}^{d^\prime}}
     \exp\biggl(-\frac{|\eta|^2}{2 v_{T,x,\xi}}\biggr)
     d\mu_T^0
    \quad \text{for }\eta\in \mathbb{R}^{d^\prime}.
\]
Then it holds
\begin{equation}\label{l.density.Y.1}
    \mathbf{E}[f(F_{T,x})\mid x+b(T)=\xi]
    =\int_{\mathbb{R}^{d^\prime}} f(\eta)q_{T,x,\xi}(\eta)
    d\eta
    \quad\text{for any bounded }
    f\in C(\mathbb{R}^{d^\prime}),
\end{equation}
where $\mathbf{E}[\,\cdot\mid x+b(T)=\xi]$ stands for the
conditional expectation given $x+b(T)=\xi$ with respect to
$\mu_T\times\widehat{\mu}_T$.
\end{lemma}

\begin{proof}
Realize the $d$-dimensional Brownian bridge
$\{\rho(t)\}_{t\in[0,T]}$ with $\rho(0)=\rho(T)=0$ by the
stochastic differential equation
\[
    d\rho(t)=db(t)-\frac{\rho(t)}{T-t}dt.
\]
Setting
$\phi(t)=|(\rho+\ell^{T,x,\xi})(t)|^\gamma I_{d^\prime}$, we
have
\[
    \mathbf{E}[f(F_{T,x})\mid b(T)=\xi]
    =\int_{\mathcal{W}_T\times\widehat{\mathcal{W}}_T}
      f\biggl(\int_0^T \phi(t) d\widehat{b}(t)\biggr)
      d(\mu_T\times\widehat{\mu}_T).
\]
Since the distribution of
$\int_0^T \phi(t) \phi(t)^\dagger dt$ under $\mu_T$ coincides
with that of $v_{T,x,\xi} I_{d^\prime}$ under $\mu_T^0$,
by Lemmas~\ref{lem.density.phi} and \ref{lem.non-deg},
we obtain the assertion.
\end{proof}

\begin{proof}[Proof of Theorem~\ref{thm.main}]
(i)
The integrability of $v_{T,x,\xi}^{-1}$ was already seen in
Lemma~\ref{lem.non-deg}.

To show \eqref{t.main.1}, realize the the diffusion process  
$\{Z^{(x,y)}(t)\}_{t\in[0,T]}$ generated by
$\frac12\Delta_{(\gamma)}$ starting at 
$(x,y)\in \mathbb{R}^d\times \mathbb{R}^{d^\prime}$
by the the It\^o type stochastic differential equation
\[
    dZ^{(x,y)}(t)=\sum_{i=1}^d V_i(Z^{(x,y)}(t))db^i(t)
         +\sum_{j=1}^{d^\prime} W_j(Z^{(x,y)}(t))
         d\widehat{b}^j(t)
    \quad\text{with }Z^{(x,y)}(0)=(x,y).
\]
Let $X^{(x,y)}(t)$ (resp. $Y^{(x,y)}(t)$) be the
$\mathbb{R}^d$-part (resp. $\mathbb{R}^{d^\prime}$-part) of 
$Z^{(x,y)}(t)$.
Then
\begin{equation}\label{eq:sol}
    X^{(x,y)}(t)=x+b(t)\quad\text{and}\quad
    Y^{(x,y)}(t)=y+\int_0^t |x+b(s)|^\gamma d\widehat{b}(s).
\end{equation}

For bounded 
$f\in C(\mathbb{R}^d\times \mathbb{R}^{d^\prime})$, 
we have 
\begin{align*}
    & \int_{\mathcal{W}_T\times \widehat{\mathcal{W}}_T} 
       f(Z^{(x,y)}(T))d(\mu_T\times\widehat{\mu}_T)
    \\
    & =\int_{\mathbb{R}^d}
      \mathbf{E}[f(\xi,Y^{(x,y)}(T))\mid x+b(T)=\xi]
      \frac1{\sqrt{2\pi T}^d} 
      \exp\Bigl(-\frac{|\xi-x|^2}{2T}\Bigr) d\xi.
\end{align*}
Due to \eqref{eq:sol}, using the function $q_{T,x,\xi}$
defined in Lemma~\ref{lem.density.Y}, we have 
\[
    \mu_T\times\widehat{\mu}_T(Z^{(x,y)}(t)\in d\xi d\eta)
    = \frac1{\sqrt{2\pi T}^d} 
      \exp\Bigl(-\frac{|\xi-x|^2}{2T}\Bigr)
      q_{T,x,\xi}(\eta-y) d\xi d\eta.
\]
This completes the proof of \eqref{t.main.1}.
\\
(ii)
Let $R>0$.
By Lemma~\ref{lem.non-deg}, the family
\[
    \biggl\{ v_{T,x,\xi}^{-\frac{d^\prime}2}
      \exp\biggl(-\frac{|\eta-y|^2}{2v_{T,x,\xi}}
        \biggr)\,;\, 
       |x|,|\xi|\le R, y,\eta\in \mathbb{R}^{d^\prime}
    \biggr\}
\]
is uniformly integrable.
Hence the mapping 
\[
    (B(R)\times \mathbb{R}^{d^\prime})^2
    \ni ((x,y),(\xi,\eta)) \mapsto
    \int_{\mathcal{W}_T^0} v_{T,x,\xi}^{-\frac{d^\prime}2}
      \exp\biggl(-\frac{|\eta-y|^2}{2v_{T,x,\xi}}
        \biggr) d\mu_T^0,
\]
where $B(R)=\{x\in \mathbb{R}^d\mid |x|\le R\}$,
is continuous.
By virtue of the expression \eqref{t.main.1},
we arrive at the desired continuity of $p_T$.
\end{proof}

In the remaining of this section, we apply
Theorem~\ref{thm.main} to the short-time on-diagonal
asymptotics of $p_T$.
Our goal will be

\begin{theorem}\label{thm.asympt.on}
Let $y\in \mathbb{R}^{d^\prime}$.
\\
{\rm (i)}
For $x\ne 0$, it holds
\begin{equation}\label{t.asympt.on.1}
    \sqrt{T}^{d+d^\prime} p_T((x,y),(x,y))
    \to \frac{1}{\sqrt{2\pi}^d |x|^{\gamma d^\prime}}
    \quad \text{as }T\to0.
\end{equation}
{\rm (ii)}
It holds
\begin{equation}\label{t.asympt.on.2}
    \sqrt{T}^{d+(1+\gamma) d^\prime} p_T((0,y),(0,y))
    =\frac1{\sqrt{2\pi}^{d+d^\prime}}
    \int_{\mathcal{W}_1^0} \biggl(\int_0^1
     |\beta(t)|^{2\gamma} dt\biggr)^{-\frac{d^\prime}2} d\mu_1^0
    \quad\text{for }T\in(0,1].
\end{equation}
\end{theorem}

Since $\int_0^1 |\beta(t)|^{2\gamma}dt=v_{1,0,0}$, by
Lemma~\ref{lem.non-deg}, the integral in 
\eqref{t.asympt.on.2} is finite.
To show the theorem, we prepare several lemmas.

\begin{lemma}\label{lem.p_T.alt}
Define the random variable $\hat{v}_{T,x,\xi}$ on
$\mathcal{W}_1^0$ by 
\[
    \hat{v}_{T,x,\xi}=\int_0^1 |(\sqrt{T} \beta
      +\ell^{1,x,\xi})(t)|^{2\gamma} dt.
\]
Then it holds
\begin{equation}\label{l.p_T.alt.1}
    p_T((x,y),(\xi,\eta))
    =\frac1{\sqrt{2\pi T}^{d+d^\prime}} 
     \exp\biggl(-\frac{|\xi-x|^2}{2T}\biggr)
     \int_{\mathcal{W}_1^0} \hat{v}_{T,x,\xi}^{-\frac{d^\prime}2}
     \exp\biggl(-\frac{|\eta-y|^2}{2T\hat{v}_{T,x,\xi}}
       \biggr) d\mu_1^0.
\end{equation}
\end{lemma}

\begin{proof}
Under $\mu_T^0$, $\{\beta(t)\}_{t\in[0,T]}$ is a
continuous Gaussian process with mean $0$ and covariance  
$\{t\wedge s-\frac{ts}{T}\}I_d$.
Since so is $\{\sqrt{T}\beta(\frac{t}{T})\}_{t\in[0,T]}$
under $\mu_1^0$, they have the same law.
Hence $v_{T,x,\xi}$ under $\mu_T^0$ and 
$T \hat{v}_{T,x,\xi}$ under $\mu_1^0$ have the same law.
Plugging this into \eqref{t.main.1}, we obtain 
\eqref{l.p_T.alt.1}.
\end{proof}

\begin{lemma}\label{lem.y=eta}
Let $y\in \mathbb{R}^{d^\prime}$.
If $(x,\xi)\ne(0,0)\in (\mathbb{R}^{d_1})^2$, then
\begin{equation}\label{l.y=eta.1}
    \sqrt{2\pi T}^{d+d^\prime} 
    \exp\biggl(\frac{|\xi-x|^2}{2T}\biggr)
    p_T((x,y),(\xi,y))
    \to\biggl(\int_0^1|\ell^{1,x,\xi}(t)|^{2\gamma} dt
         \biggr)^{-\frac{d^\prime}2}.
\end{equation}
\end{lemma}

\begin{proof}
Suppose $(x,\xi)\ne(0,0)$.
Define 
\[
    B_{T,x,\xi}^\prime
    =\|\sqrt{T}\beta+\ell^{1,x,\xi}\|_{1,12,\frac14}.
\]
Since 
$B_{T,x,\xi}^\prime\le
 \sqrt{T}B_{1,0,0}+\|\ell^{1,x,\xi}\|_{1,12,\frac14}$, 
by \eqref{eq:exp.int}, we obtain
\begin{equation}\label{l.y=eta.21}
    \sup_{T\in[0,1]}\|B_{T,x,\xi}^\prime\|_{L^p(\mu_1^0)}
    <\infty\quad\text{for any }p\in(1,\infty).
\end{equation}

If $x\ne0$, then, in repetition of the argument used to show  
\eqref{l.non-deg.32}, we have
\[
    \hat{v}_{T,x,\xi}\ge \Bigl(\frac{|x|}2\Bigr)^{2\gamma}
     \biggl(\frac{|x|^6}{192^6(B_{T,x,\xi}^\prime)^6}\wedge1
       \biggr).
\]
If $\xi\ne0$, then by \eqref{eq:garsia} with $s=1$, it holds
\[
    |(\sqrt{T}\beta+\ell^{1,x,\xi})(t)|
    \ge\frac{|\xi|}2
    \quad\text{for $t\in[0,1]$ satisfying }
    96 B_{T,x,\xi}^\prime|1-t|^{\frac16}\le \frac{|\xi|}2.
\]
Hence, for $\xi\ne0$, we have
\[
    \hat{v}_{T,x,\xi}
    \ge \Bigl(\frac{|\xi|}2\Bigr)^{2\gamma}
     \biggl(\frac{|\xi|^6}{192^6(B_{T,x,\xi}^\prime)^6}\wedge1
       \biggr).
\]
Thus we obtain
\[
    \hat{v}_{T,x,\xi}
    \ge \Bigl(\frac{|x|\vee|\xi|}2\Bigr)^{2\gamma}
     \biggl(\frac{(|x|\vee|\xi|)^6}{
           192^6(B_{T,x,\xi}^\prime)^6}\wedge1 \biggr).
\]
In conjunction with \eqref{l.y=eta.21}, this yields
\begin{equation}\label{l.y=eta.22}
    \sup_{T\in(0,1]}
    \|\hat{v}_{T,x,\xi}^{-1}\|_{L^p(\mu_1^0)}
    <\infty\quad\text{for any }p\in(1,\infty).
\end{equation}

By Lemma~\ref{lem.p_T.alt}, it holds
\begin{equation}\label{l.y=eta.23}
    \sqrt{2\pi T}^{d+d^\prime}
    \exp\biggl(\frac{|\xi-x|^2}{2T}\biggr)
    p_T((x,y),(\xi,y))
    =\int_{\mathcal{W}_1^0} \hat{v}_{T,x,\xi}^{-\frac{d^\prime}2} 
      d\mu_1^0.
\end{equation}
Since the family
$\{\hat{v}_{T,x,\xi}^{-\frac{d^\prime}2}:T\in(0,1] \}$ is uniformly
integrable by \eqref{l.y=eta.22}, and 
$\hat{v}_{T,x,\xi}$ converges to 
$\int_0^1|\ell^{1,x,\xi}(t)|^{2\gamma}dt$ point-wise as $T\to0$,
we see 
\[
    \int_{\mathcal{W}_T^d} \hat{v}_{T,x,\xi}^{-\frac{d^\prime}2} 
      d\mu_1^0
    \to \biggl(\int_0^1|\ell^{1,x,\xi}(t)|^{2\gamma}dt
            \biggr)^{-\frac{d^\prime}2}.
\]
Plugging this into \eqref{l.y=eta.23}, we obtain
\eqref{l.y=eta.1}.
\end{proof}

\begin{proof}[Proof of Theorem~\ref{thm.asympt.on}]
The convergence  \eqref{t.asympt.on.1} was already seen 
in Lemma~\ref{lem.y=eta}, for $\ell^{1,x,x}\equiv x$.
Since 
$\hat{v}_{T,0,0}=T^\gamma
     \int_0^1|\beta(t)|^{2\gamma}dt$,
\eqref{t.asympt.on.2} follows from \eqref{l.p_T.alt.1}.
\end{proof}

\section{Off-diagonal asymptotics}\label{sec:off-diagonal}

\subsection{When $(x,\xi)\ne(0,0)$}

The aim of this subsection is to obtain the off-diagonal
asymptotics of $p_T((x,y),(\xi,\eta))$ as $T\to0$ when 
$(x,\xi)\ne(0,0)$.
To do so, we will apply the theory of large deviations.

We first introduce quantities used in the expression of the
asymptotics.
For $x,\xi\in \mathbb{R}^d$, $w\in \mathcal{W}_1^0$,
$h\in \mathcal{H}_1^0$, and $a\ge0$, put
\begin{equation}\label{t.asympt.off.20}
   \boldsymbol{v}_{x,\xi}(w)
     =\int_0^1 |w(t)+\ell^{1,x,\xi}(t)|^{2\gamma} dt,
   \quad
   \Phi_{x,\xi,a}(h)
      =\frac{a^2}{\boldsymbol{v}_{x,\xi}(h)}
       +\|h\|_{\mathcal{H}_1^0}^2.
\end{equation}
If $w+\ell^{1,x,\xi}\ne0$, then
$\boldsymbol{v}_{x,\xi}(w)>0$.
Hence $\boldsymbol{v}_{x,\xi}(w)>0$ for any 
$w\in \mathcal{W}_1^0$ when $(x,\xi)\ne(0,0)$, and 
$\boldsymbol{v}_{0,0}(w)>0$ for $w\ne 0$.
Define
\[
    m(x,\xi,a)=\inf_{h\in \mathcal{H}_1^0} \Phi_{x,\xi,a}(h).
\]
We first see the properties of $m(x,\xi,a)$;

\begin{lemma}\label{lem.m()}
{\rm(i)} 
$m(x,\xi,0)=0$, 
the function $[0,\infty)\ni a\mapsto m(x,\xi,a)$ is 
non-decreasing, and it holds
\[
    (a\wedge1)^2m(x,\xi,1) \le m(x,\xi,a)
    \le \frac{a^2}{\boldsymbol{v}_{x,\xi}(0)}
    \quad\text{for any }a\ge0.
\]
{\rm(ii)} 
It holds
\begin{equation}\label{l.m().1}
    \lim_{a\to\infty} a^{-\frac{2}{1+\gamma}}
         m(x,\xi,a)
    =c_\gamma^{-\frac{2\gamma}{1+\gamma}}(1+\gamma)
        \gamma^{-\frac{\gamma}{1+\gamma}},
\end{equation}
where 
\[
    c_\gamma=\sup_{h\in \mathcal{H}_1^0}
         \frac{\|h\|_{2\gamma}}{\|h\|_{\mathcal{H}_1^0}}
    \quad\text{with }
    \|h\|_{2\gamma}
    =\biggl(\int_0^1 |h(t)|^{2\gamma} dt
       \biggr)^{\frac1{2\gamma}}.
\]
{\rm(iii)}
For each $a>0$, there exists $h_a\in \mathcal{H}_1^0$ such
that $m(x,\xi,a)=\Phi_{x,\xi,a}(h_a)$.
\\
{\rm(iv)}
$m(x,\xi,a)>0$ for $a>0$.
\end{lemma}

\begin{proof}
(i)
The assertion is an immediate consequence of the definition
of $m(x,\xi,a)$.
\\
(ii) For $p,q>0$ and $r\ge0$, set
\[
    \Psi_{p,q,r}(\lambda)
    =\begin{cases}
      \dfrac{p}{\lambda^{2\gamma}+r}
      +q\lambda^2 
      & \text{when }\gamma<\dfrac12,
     \\[10pt]
      \dfrac{p}{(\lambda+r)^{2\gamma}}
      +q\lambda^2
      & \text{when }\gamma\ge\dfrac12,
     \end{cases}
    \quad\text{for }\lambda\in(0,\infty).
\]
It is easily seen that
\begin{equation}\label{l.m().101}
    \inf_{\lambda\in(0,\infty)} \Psi_{p,q,r}(\lambda)
    =q\Bigl(1+\frac1{\gamma}\Bigr)\lambda_{p,q,r}^2
     +\frac{qr}{\gamma}\lambda_{p,q,r}^{1+(1-2\gamma)\vee0},
\end{equation}
where $\lambda_{p,q,r}>0$ is determined by
\[
    q=\begin{cases}
        \dfrac{\gamma p}{\lambda_{p,q,r}^{2-2\gamma}
           (\lambda_{p,q,r}^{2\gamma}+r)^2}
        & \text{when }\gamma<\dfrac12,
        \\[10pt]
        \dfrac{\gamma p}{
           (\lambda_{p,q,r}+r)^{2\gamma+1}}
        & \text{when }\gamma\ge\dfrac12.
      \end{cases}
\]
The definition of $\lambda_{p,q,r}$ yields
\[
    \lim_{r\to0}\lambda_{p,q,r}=\lambda_{p,q,0}
    =\Bigl(\frac{\gamma p}{q}\Bigr)^{\frac1{2+2\gamma}}.
\]
Plugging this into \eqref{l.m().101}, we obtain
\begin{equation}\label{l.m().104}
    \lim_{r\to0}
    \inf_{\lambda\in(0,\infty)}\Psi_{p,q,r}(\lambda)
    =\inf_{\lambda\in(0,\infty)}\Psi_{p,q,0}(\lambda)
    =(1+\gamma)\gamma^{-\frac{\gamma}{1+\gamma}}
     (pq^\gamma)^{\frac1{1+\gamma}}.
\end{equation}

Set $\alpha=(1+\gamma)^{-1}$ and observe 
\[
    \Phi_{x,\xi,a}(a^\alpha h)
    =a^{2 \alpha} \Phi_{a^{-\alpha}x,a^{-\alpha}\xi,1}(h)
    \quad\text{for any }h\in \mathcal{H}_1^0.
\]
Thus, to show \eqref{l.m().1}, it suffices to prove
\begin{align}
    &  \liminf_{a\to\infty} 
         m(a^{-\alpha}x,a^{-\alpha}\xi,1)
       \ge c_\gamma^{-\frac{2\gamma}{1+\gamma}}(1+\gamma)
        \gamma^{-\frac{\gamma}{1+\gamma}},
 \label{l.m().105}
    \\
    &  \limsup_{a\to\infty} 
         m(a^{-\alpha}x,a^{-\alpha}\xi,1)
       \le c_\gamma^{-\frac{2\gamma}{1+\gamma}}(1+\gamma)
        \gamma^{-\frac{\gamma}{1+\gamma}}.
 \label{l.m().106}
\end{align}

We first show \eqref{l.m().105}.
Applying the inequality 
$(a+b)^{2\gamma}\le a^{2\gamma}+b^{2\gamma}$ when 
$\gamma<\frac12$ and the Minkowski inequality when
$\gamma\ge\frac12$, we have 
\[
    \int_0^1|h(t)+
       \ell^{1,a^{-\alpha}x,a^{-\alpha}\xi}(t)|^{2\gamma} dt 
    \le 
     \begin{cases}
       c_\gamma^{2\gamma}\|h\|_{\mathcal{H}_1^0}^{2\gamma}
       +r(a)^{2\gamma}
      & \text{when }\gamma<\dfrac12,
      \\[10pt]
       \bigl(c_\gamma\|h\|_{\mathcal{H}_1^0}+r(a)\bigr)^{2\gamma}
      & \text{when }\gamma\ge\dfrac12
     \end{cases}
     \quad\text{for any }h\in \mathcal{H}_1^0,
\]
where
$r(a)=\|\ell^{1,a^{-\alpha}x,a^{-\alpha}\xi}\|_{2\gamma}$.
Hence
\[
    \Phi_{a^{-\alpha}x,a^{-\alpha}\xi,1}(h)
    \ge 
      \begin{cases}
        \inf_{\lambda\in(0,\infty)}
          \Psi_{c_\gamma^{-2\gamma},1,r(a)^{2\gamma}}(\lambda)
        & \text{when }\gamma<\dfrac12,
        \\[10pt]
        \inf_{\lambda\in(0,\infty)}
          \Psi_{c_\gamma^{-2\gamma},1,r(a)}(\lambda)
       & \text{when }\gamma\ge\dfrac12
      \end{cases}
     \quad\text{for any }h\in \mathcal{H}_1^0.
\]
Since $\lim_{a\to\infty}r(a)=0$,
in conjunction with \eqref{l.m().104}, this implies
\eqref{l.m().105}.

We next show \eqref{l.m().106}.
It holds
\[
    m(a^{-\alpha}x,a^{-\alpha}\xi,1)
    \le \Phi_{a^{-\alpha}x,a^{-\alpha}\xi,1}(h)
    \overset{a\to\infty}{\longrightarrow}
    \Phi_{0,0,1}(h) 
    \quad\text{for any }h\in \mathcal{H}_1^0.
\]
Hence we have
\[
    \limsup_{a\to\infty}
       m(a^{-\alpha}x,a^{-\alpha}\xi,1)
    \le \inf_{\lambda\in(0,\infty)} \Phi_{0,0,1}(\lambda h)
    =\inf_{\lambda\in(0,\infty)} \Psi_{p(h),q(h),0}(\lambda)
    \quad\text{for any }h\in \mathcal{H}_1^0,
\]
where 
$p(h)=\|h\|_{2\gamma}^{2\gamma}$ and
$q(h)=\|h\|_{\mathcal{H}_1^0}^2$.
By \eqref{l.m().104}, 
\[
    \limsup_{a\to\infty}
       m(a^{-\alpha}x,a^{-\alpha}\xi,1)
    \le (1+\gamma)\gamma^{-\frac{\gamma}{1+\gamma}}
        (p(h)q(h)^\gamma)^{\frac1{1+\gamma}}.    
\]
Taking the infimum over $h$ and noticing
\[
    \inf_{h\in \mathcal{H}_1^0} p(h)q(h)^\gamma
    =c_\gamma^{-2\gamma},
\]
we arrive at \eqref{l.m().106}.
\\
(iii)
Take a sequence $\{h_n\}_{n=1}^\infty$ such that
$\Phi_{x,\xi,a}(h_n)\to m(x,\xi,a)$ as $n\to\infty$.
Since $\|h\|_{\mathcal{H}_1^0}^2\le \Phi_{x,\xi,a}(h)$ for 
$h\in \mathcal{H}_1^0$, 
\begin{equation}\label{l.m().21}
    \sup_{n\in \mathbb{N}} \|h_n\|_{\mathcal{H}_1^0}<\infty.
\end{equation}
Hence taking a subsequence if necessary, we may assume $h_n$
converges weakly to some $h_a\in \mathcal{H}_1^0$. 

By \eqref{eq:h-conti}, \eqref{l.m().21}, and the
Ascoli-Arzel\`{a} theorem, 
$\{h_n\}_{n=1}^\infty$ have a subsequence
$\{h_{n_j}\}_{j=1}^\infty$ which converges in $\mathcal{W}_1^0$.
Using the inclusion 
$\mathcal{W}_1^{0*}\subset \mathcal{H}_1^0$, we observe that 
\[
    \ell(h_{n_j})
    =\langle \ell,h_{n_j}\rangle_{\mathcal{H}_1^0}
     \overset{j\to\infty}{\longrightarrow}
      \langle \ell,h_a\rangle_{\mathcal{H}_1^0}
    =\ell(h_a)
    \quad\text{for any }\ell\in \mathcal{W}_1^{0*},
\]
which means $h_{n_j}$ converges to $h_a$ weakly in
$\mathcal{W}_1^0$.
Hence $\|h_{n_j}-h_a\|_{\mathcal{W}_T^0}\to0$ as $j\to\infty$.
In particular, 
$\boldsymbol{v}_{x,\xi}(h_{n_j})
 \to\boldsymbol{v}_{x,\xi}(h_a)$ as $j\to\infty$.

Since
$\|h_a\|_{\mathcal{H}_1^0}
 \le\liminf_{j\to\infty}\|h_{n_j}\|_{\mathcal{H}_1^0}$,
we obtain
\[
    \Phi_{x,\xi,a}(h_a)\le \liminf_{j\to\infty}
        \Phi_{x,\xi,a}(h_{n_j})=m(x,\xi,a).
\]
This implies the identity $\Phi_{x,\xi,a}(h_a)=m(x,\xi,a)$.
\\
(ii)
Since $\boldsymbol{v}_{x,\xi}(h_a)>0$,
$m(x,\xi,a)=\Phi_{x,\xi,a}(h_a)
  \ge \frac{a^2}{\boldsymbol{v}_{x,\xi}(h_a)}>0$.
\end{proof}

\begin{remark}
Let $h\in \mathcal{H}_1^0$ and $t\in[0,1]$.
Since 
$h(t)=\int_0^1 (\boldsymbol{1}_{[0,t]}(s)-t)\dot{h}(s)ds$,
we have 
$|h(t)|\le (t-t^2)^{\frac12}\|h\|_{\mathcal{H}_1^0}$.
Hence 
\[
    c_\gamma\le B(1+\gamma,1+\gamma)^{\frac1{2\gamma}},
\]
where $B(\cdot,\cdot)$ is the Beta function.
\end{remark}

The goal of this subsection is

\begin{theorem}\label{thm.asympt.off}
Assume $(x,\xi)\ne(0,0)$.
Then it holds
\begin{equation}\label{t.asympt.off.1}
    \lim_{T\to0} T\log p_T((x,y),(\xi,\eta))
    =-\frac12\{|\xi-x|^2+m(x,\xi,|\eta-y|)\}
    \quad\text{for any }y,\eta\in\mathbb{R}^{d^\prime}.
\end{equation}
Moreover, 
\begin{equation}\label{t.asympt.off.1+}
    \lim_{|\eta-y|\to\infty}
     |\eta-y|^{-\frac2{1+\gamma}} 
       \lim_{T\to0} T\log p_T((x,y),(\xi,\eta))
     =c_\gamma^{-\frac{2\gamma}{1+\gamma}}(1+\gamma)
      \gamma^{-\frac{\gamma}{1+\gamma}},
\end{equation}
where $c_\gamma$ is the constant given in \eqref{l.m().1}.
\end{theorem}

To show the theorem, we recall the results in the theory of
large deviations.
Define 
$\mathfrak{z}_T:\mathcal{W}_1^0\to \mathcal{W}_1^0$
by $\mathfrak{z}_T(w)=\sqrt{T}w$ for 
$w\in \mathcal{W}_1^0$, and set
$\nu_T=\mu_1^0\circ \mathfrak{z}_T^{-1}$.
Then the Schilder theorem and the Varadhan integral lemma 
hold on $\mathcal{W}_1^0$ as follows.

\begin{lemma}\label{lem.varadhan}
{\rm (i)}
$\{\nu_T\}$ satisfies the large deviation principle with
rate function
\[
    I(w)=\begin{cases}\displaystyle
            \frac12\|w\|_{\mathcal{H}_0}^2
               & \text{if }w\in \mathcal{H}_0,
            \\
            \infty & \text{if }w\in 
              \mathcal{W}_1^0\setminus \mathcal{H}_0.
         \end{cases}
\]
{\rm (ii)}
If $F\in C(\mathcal{W}_1^0)$ is bounded from above, then
\begin{equation}\label{l.varadhan.1}
    \lim_{T\to0} T\log \int_{\mathcal{W}_1^0} 
      \exp\Bigl(\frac1{T}F\Bigr) d\nu_T
    =\sup_{w\in \mathcal{W}_1^0}\{F(w)-I(w)\}.
\end{equation}
\end{lemma}

\begin{proof}
See \cite[Chapter~8]{stroock} for (i) and 
\cite[Theorem~4.3.1]{dembo_zeitouni} for (ii).
\end{proof}

\begin{proof}[Proof of Theorem~\ref{thm.asympt.off}]
Notice that
$\boldsymbol{v}_{x,\xi}(\mathfrak{z}_T)=\hat{v}_{T,x,\xi}$.
By \eqref{l.p_T.alt.1}, we obtain
\begin{equation}\label{eq:p_T.ldp}
    p_T((x,y),(\xi,\eta))
    =\frac1{\sqrt{2\pi T}^{d+d^\prime}} 
     \exp\biggl(-\frac{|\xi-x|^2}{2T}\biggr)
     \mathcal{I}_{T,x,\xi},
\end{equation}
where
\[
     \mathcal{I}_{T,x,\xi}
      =\int_{\mathcal{W}_1^0}  
        \boldsymbol{v}_{x,\xi}^{-\frac{d^\prime}2}
         \exp\biggl(
         -\frac{|\eta-y|^2}{2T \boldsymbol{v}_{x,\xi}}
         \biggr) d\nu_T.
\]
Thus, to see \eqref{t.asympt.off.1}, it suffices to show
\[
    \lim_{T\to0} T\log \mathcal{I}_{T,x,\xi}
    =-\frac12 m(x,\xi,|\eta-y|).
\]
We prove this by showing the upper and
lower estimations
\begin{align}
    & \limsup_{T\to0} T\log \mathcal{I}_{T,x,\xi}
      \le -\frac12 m(x,\xi,|\eta-y|),
 \label{t.asympt.off.31}
    \\
    & \liminf_{T\to0} T\log\mathcal{I}_{T,x,\xi}
      \ge -\frac12 m(x,\xi,|\eta-y|).
 \label{t.asympt.off.32}
\end{align}

Since $\boldsymbol{v}_{x,\xi}>0$,
$\boldsymbol{v}_{x,\xi}^{-1}\in C(\mathcal{W}_1^0)$.
Moreover, it holds
\[
    \sup_{w\in \mathcal{W}_1^0} \biggl\{
      -\frac{a^2}{2 \boldsymbol{v}_{x,\xi}(w)}
       -I(w)\biggr\}
    =-\frac12 m(x,\xi,a).
\]
By Lemma~\ref{lem.varadhan}(ii), we then have
\begin{equation}\label{t.asympt.off.22}
    \lim_{T\to0} T\log \int_{\mathcal{W}_1^0} \exp\biggl(
       -\frac{a^2}{2T \boldsymbol{v}_{x,\xi}}\biggr)
       d\nu_T
    =-\frac12 m(x,\xi,a)
    \quad\text{for }a\ge0.
\end{equation}

For $p\in(1,\infty)$, 
by the H\"older inequality, \eqref{l.y=eta.22}, 
\eqref{t.asympt.off.22}, and Lemma~\ref{lem.m()}, we obtain
\[
    \limsup_{T\to0} T\log \mathcal{I}_{T,x,\xi}
    \le -\frac1{2p} m(x,\xi,\sqrt{p}|\eta-y|)
    \le -\frac1{2p} m(x,\xi,|\eta-y|).
\]
Letting $p\searrow1$, we obtain \eqref{t.asympt.off.31}.

Let $R>\sqrt{m(x,\xi,|\eta-y|)}$ and 
$\mathcal{B}_R=\{w\in \mathcal{W}_1^0
     \mid\|w\|_{\mathcal{W}_1^0}<R\}$.
Noting that
$\sup_{w\in \mathcal{B}_R} \boldsymbol{v}_{x,\xi}(w)
    \le (R+|x|+|\xi|)^{2\gamma}$,
we have
\begin{align}
    \mathcal{I}_{T,x,\xi}
    & \ge (R+|x|+|\xi|)^{-\gamma d^\prime}
       \int_{\mathcal{B}_R} \exp\biggl(
      -\frac{|\eta-y|^2}{2T \boldsymbol{v}_{x,\xi}}\biggr)
        d\nu_T
 \nonumber
    \\
    & \ge (R+|x|+|\xi|)^{-\gamma d^\prime} \biggl\{
        \int_{\mathcal{W}_0^1} \exp\biggl(
           -\frac{|\eta-y|^2}{2T \boldsymbol{v}_{x,\xi}}
           \biggr) d\nu_T
       -\nu_T(\mathcal{B}_R^c)\biggr\}.
 \label{t.asympt.off.23}
\end{align}
Since $\|h\|_{\mathcal{W}_1^0}\le \|h\|_{\mathcal{H}_1^0}$
for $h\in \mathcal{H}_1^0$, by Lemma~\ref{lem.varadhan}(i),
we obtain
\[
    \limsup_{T\to0} T\log\nu_T(\mathcal{B}_R^c)
    \le -\inf_{w\in \mathcal{B}_R^c} I(w)\le -\frac{R^2}2.
\]
In conjunction with this and \eqref{t.asympt.off.22}, 
the inequality \eqref{t.asympt.off.23} yields
\eqref{t.asympt.off.31}.

The equation \eqref{t.asympt.off.1+} follows from
Lemma~\ref{lem.m()} and \eqref{t.asympt.off.1}.
\end{proof}

\subsection{When $(x,\xi)=(0,0)$}

In this subsection,  we investigate the off-diagonal
asymptotics of $p_T((0,y),(0,\eta))$ as $T\to0$.
Set 
\[
    \delta_0=\sup\{\delta>0\mid 
         \exp(\delta B_{1,0,0}^2)\in L^1(\mu_1^0)\}.
\]
By \eqref{eq:exp.int} with $T=1$, $\delta_0>0$.

\begin{theorem}\label{thm.asympt.off.00}
For $\delta>0$ define $\varepsilon(\delta)>0$ by
\[
    \varepsilon(\delta)^2=\frac{\delta}{8(192)^2}.
\]
Then it holds
\begin{align}
    & -2^{\frac{3\gamma}{1+\gamma}}(1+\gamma)
    \gamma^{-\frac{\gamma}{1+\gamma}} \Bigl(
      (2 \varepsilon(\delta_0))^6\wedge1
      \Bigr)^{-\frac1{1+\gamma}}
      |\eta-y|^{\frac2{1+\gamma}}
 \label{t.asympt.off.2}
    \\
    & \qquad
    \le \liminf_{T\to0} T\log p_T((0,y),(0,\eta))
 \nonumber
    \\
    & \qquad
    \le  \limsup_{T\to0} T\log p_T((0,y),(0,\eta)) 
    \le  -2^{-\frac{1-\gamma}{1+\gamma}} 
          d^{-\frac{\gamma}{1+\gamma}}
          |\eta-y|^{\frac2{1+\gamma}} 
 \label{t.asympt.off.3}
\end{align}
for any $y,\eta\in\mathbb{R}^{d^\prime}$.
\end{theorem}

The dependence on the parameter $\gamma$ can be seen in
these estimations, especially in the growth order in
$|\eta-y|$ as it tends to infinity, which is similar as in
the case when $(x,\xi)\ne(0,0)$
(cf. Theorem~\ref{thm.asympt.off}).
The proof is divided into several steps, each step being a
lemma.
In what follows, let $M_{1,0}$ be the random variable given
in Lemma~\ref{lem.max.prob};
$M_{1,0}(w)=\|w\|_{\mathcal{W}_1^0}$ for 
$w\in\mathcal{W}_1^0$.

\begin{lemma}\label{lem.prob}
It holds
\begin{align}
    & \mu_1^0( {M}_{1,0}\ge a)
      \le 2d\exp\Bigl(-\frac{2a^2}{d}\Bigr),
 \label{l.prob.1}
    \\
    & \mu_1^0(a\le  {M}_{1,0}\le 2\sqrt{d} a)
      \ge 2\exp(-2a^2)\{1-(d+1)\exp(-6a^2)\}
      \quad\text{for }a>0.
 \label{l.prob.2}
\end{align}
\end{lemma}

\begin{proof}
It is known (\cite[p.39 (12)]{shorack_wellner}) that
\[
    \mu_1^0\Bigl(\max_{t\in[0,1]}|\beta^i(t)|\ge a\Bigr)
    =2\sum_{k=1}^\infty (-1)^{k+1} \exp(-2k^2a^2)
    \quad\text{for $1\le i\le d$ and }a>0,
\]
where $\beta^i(t)$ is the $i$th component of $\beta(t)$.
Since $k\mapsto \exp(-k^2a^2)$ is non-increasing, we have
\[
    2\exp(-2a^2)\{1-\exp(-6a^2)\} \le
    \mu_1^0\Bigl(\max_{t\in[0,1]}|\beta^i(t)|\ge a\Bigr)
    \le 2\exp(-2a^2).
\]
Combining this with the inclusion
\[
    \Bigl\{\max_{t\in[0,1]}|\beta^1(t)|\ge a\Bigr\}
    \subset
    \{ {M}_{1,0}\ge a\}
    \subset
    \bigcup_{i=1}^{d} \Bigl\{\max_{t\in[0,1]}|\beta^i(t)|
        \ge \frac{a}{\sqrt{d}}\Bigr\},
\]
we obtain \eqref{l.prob.1} and
\[
    2\exp(-2a^2)\{1-\exp(-6a^2)\} 
    \le \mu_1^0( {M}_{1,0}\ge a).
\]
This and \eqref{l.prob.1} with $2\sqrt{d}a$ for $a$ yield 
\eqref{l.prob.2}.
\end{proof}

\begin{lemma}\label{lem.upper}
It holds
\begin{equation}\label{l.upper.1}
    \limsup_{T\to0} T\log p_T((0,y),(0,\eta))
    \le -2^{-\frac{1-\gamma}{1+\gamma}} d^{-\frac{\gamma}{1+\gamma}}
      |\eta-y|^{\frac2{1+\gamma}}.
\end{equation}
In particular, the upper estimate \eqref{t.asympt.off.3} in
Theorem~\ref{thm.asympt.off.00} holds.
\end{lemma}

\begin{proof}
Due to Theorem~\ref{thm.asympt.on}, \eqref{l.upper.1} holds
if $y=\eta$.
Thus we assume $y\ne\eta$.

By \eqref{eq:p_T.ldp} and the very definition of $\nu_T$, we 
obtain
\begin{equation}\label{l.upper.21}
    p_T((0,y),(0,\eta))
    =\frac{T^{-\frac{d+d^\prime+\gamma d^\prime}2}}{\sqrt{2\pi}^{d+d^\prime}}
     \int_{\mathcal{W}_1^0} 
          \boldsymbol{v}_{0,0}^{-\frac{d^\prime}2}
     \exp\biggl(-\frac{|\eta-y|^2}{
           2T^{1+\gamma} \boldsymbol{v}_{0,0}} \biggr)
     d\mu_1^0.
\end{equation}
Let $p>1$. 
By H\"older's inequality, we have
\begin{equation}\label{l.upper.22}
    \limsup_{T\to0} T\log p_T((0,y),(0,\eta))
    \le \frac1p \limsup_{T\to0} T\log\biggl(
    \int_{\mathcal{W}_1^0} \exp\biggl(-\frac{p|\eta-y|^2}{
           2T^{1+\gamma} \boldsymbol{v}_{0,0}} \biggr) d\mu_1^0
    \biggr).
\end{equation}

Since $\boldsymbol{v}_{0,0}\le  {M}_{1,0}^{2\gamma}$,
it holds
\[
    \int_{\mathcal{W}_1^0} \exp\biggl(-\frac{p|\eta-y|^2}{
           2T^{1+\gamma} \boldsymbol{v}_{0,0}} \biggr) d\mu_1^0
    \le 
    \int_{\mathcal{W}_1^0} \exp\biggl(-\frac{p|\eta-y|^2}{
           2T^{1+\gamma}  {M}_{1,0}^{2\gamma}} \biggr) d\mu_1^0.
\]
For $\lambda>0$, using the decomposition 
$\mathcal{W}_1^0=\{ {M}_{1,0}<\lambda T^{-\frac12}\}
  \cup\{ {M}_{1,0}\ge \lambda T^{-\frac12}\}$ 
and Lemma~\ref{lem.prob}, we obtain
\[
    \int_{\mathcal{W}_1^0} \exp\biggl(-\frac{p|\eta-y|^2}{
           2T^{1+\gamma} \boldsymbol{v}_{0,0}} \biggr) d\mu_1^0
    \le \exp\biggl(
         -\frac{p|\eta-y|^2}{2T\lambda^{2\gamma}}
        \biggr)
        +2d\exp\biggl(-\frac{2\lambda^2}{dT}\biggr).
\]
Substituting this into \eqref{l.upper.22}, we obtain
\begin{align*}
    \limsup_{T\to0} T\log p_T((0,y),(0,\eta))
    & \le -\frac1p \sup_{\lambda>0}\biggl(
         \frac{p|\eta-y|^2}{2\lambda^{2\gamma}}
         \wedge
         \frac{2\lambda^2}{d}\biggr)
    \\
    & =-p^{\frac1{2+2\gamma}} 
     2^{-\frac{\gamma}{1+\gamma}}d^{-\frac1{2+2\gamma}}
      |\eta-y|^{\frac2{1+\gamma}}.
\end{align*}
Letting $p\searrow 1$, we obtain \eqref{l.upper.1}.
\end{proof}

\begin{lemma}\label{lem.lower}
It holds
\begin{equation}\label{l.lower.1}
    \liminf_{T\to0} T\log p_T((0,y),(0,\eta))
    \ge 
    -2^{\frac{3\gamma}{1+\gamma}}(1+\gamma)
    \gamma^{-\frac{\gamma}{1+\gamma}} \Bigl(
      (2 \varepsilon(\delta_0))^6\wedge1
      \Bigr)^{-\frac1{1+\gamma}}
      |\eta-y|^{\frac2{1+\gamma}}.
\end{equation}
In particular, the lower estimate \eqref{t.asympt.off.2} in
Theorem~\ref{thm.asympt.off.00} holds.
\end{lemma}

\begin{proof}
Due to Theorem~\ref{thm.asympt.on}, \eqref{l.lower.1} holds
if $y=\eta$.
Thus we assume $y\ne\eta$.

Let $\delta\in(0,\delta_0)$.
For $\lambda, T>0$, put
\[
    A_{\lambda,T}=\biggl\{B_{1,0,0}\le 
        \frac{\lambda T^{-\frac12}}{
               192 \varepsilon(\delta)} \biggr\}.
\]
It holds
\begin{equation}\label{l.lower.22}
    \mu_1^0(A_{\lambda,T}^c)
    \le C_\delta \exp\Bigl(-\frac{8\lambda^2}{T}\Bigr),
\end{equation}
where
\[
    C_\delta=\int_{\mathcal{W}_1^0} \exp(\delta B_{1,0,0}^2)
      d\mu_1^0<\infty.
\]

In repetition of the argument used to show
\eqref{l.non-deg.30}, this time with $x=\xi=0$, we can show
\[
    \boldsymbol{v}_{0,0}
    \ge \Bigl(\frac{ {M}_{1,0}}2\Bigr)^{2\gamma}
        \biggl(\frac{ {M}_{1,0}^6}{192^6B_{1,0,0}^6}
        \wedge\frac12\biggr).
\]
Thus
\[
    \boldsymbol{v}_{0,0}\boldsymbol{1}_{A_{\lambda,T}}
    \ge \Bigl(\frac{{M}_{1,0}}2\Bigr)^{2\gamma}
        \biggl(\Bigl({M}_{1,0}
            T^{\frac12}\lambda^{-1}
           \varepsilon(\delta)\Bigr)^6
           \wedge\frac12\biggr)
        \boldsymbol{1}_{A_{\lambda,T}}.
\]
By this, setting 
\[
    B_{\lambda,T}=\{\lambda T^{-\frac12}\le  
          {M}_{1,0}\le 2\sqrt{d}\lambda T^{-\frac12}\}
\]
and remembering 
$\boldsymbol{v}_{0,0}\le {M}_{1,0}^{2\gamma}$, we obtain
\begin{align}
    & \int_{\mathcal{W}_1^0} \boldsymbol{v}_{0,0}^{-\frac{d^\prime}2}
      \exp\biggl(-\frac{|\eta-y|^2}{
             2T^{1+\gamma}\boldsymbol{v}_{0,0}}\biggr)d\mu_1^0
 \nonumber
    \\
    & \quad
      \ge 
      \int_{A_{\lambda,T}\cap B_{\lambda,T}}
            {M}_{1,0}^{-\gamma d^\prime} \exp\biggl(
         -\frac{2^{2\gamma}|\eta-y|^2}{
             T^{1+\gamma} {M}_{1,0}^{2\gamma}
             (\{2( {M}_{1,0}
                T^{\frac12}\lambda^{-1}
                \varepsilon(\delta))^6\}\wedge1)}
        \biggr) d\mu_1^0
 \nonumber
    \\
    & \quad
      \ge (2\sqrt{d}\,\lambda)^{-\gamma d^\prime}
          T^{\frac{\gamma d^\prime}2}
          \exp\biggl(-\frac{2^{2\gamma}|\eta-y|^2}{
             T\lambda^{2\gamma}((2\varepsilon(\delta)^6)\wedge1)}
          \biggr)
          \mu_1^0(A_{\lambda,T}\cap B_{\lambda,T}).
 \label{l.lower.23}
\end{align}
By Lemma~\ref{lem.prob} and \eqref{l.lower.22}, we have
\begin{align*}
    \mu_1^0(A_{\lambda,T}\cap B_{\lambda,T})
    &\ge \mu_1^0(B_{\lambda,T})
         -\mu_1^0(A_{\lambda,T}^c)
    \\
    &\ge \exp\Bigl(-\frac{2\lambda^2}{T}\Bigr)
         \biggl\{2-(2+2d+C_\delta)
         \exp\Bigl(-\frac{6\lambda^2}{T}\Bigr)
         \biggr\}.
\end{align*}
Plugging this into \eqref{l.lower.23}, we obtain
\[
    \liminf_{T\to0} T \log\biggl(
     \int_{\mathcal{W}_1^0} \boldsymbol{v}_{0,0}^{-\frac{d^\prime}2}
      \exp\biggl(-\frac{|\eta-y|^2}{
             2T^{1+\gamma}\boldsymbol{v}_{0,0}}\biggr)
      d\mu_1^0\biggr)
     \ge -\inf_{\lambda\in(0,\infty)}
         \biggl\{\frac{2^{2\gamma}|\eta-y|^2}{
           \lambda^{2\gamma}\{(2\varepsilon(\delta)^6)
           \wedge1\}}
          +2\lambda^2\biggr\}.
\]
Due to the observation on $\Psi_{p,q,r}$ in the proof of
Lemma~\ref{lem.m()}, we see 
\[
    \inf_{\lambda\in(0,\infty)}
         \biggl\{\frac{2^{2\gamma}|\eta-y|^2}{
           \lambda^{2\gamma}\{(2\varepsilon(\delta)^6)
           \wedge1\}}
          +2\lambda^2\biggr\}
    =2^{\frac{3\gamma}{1+\gamma}}(1+\gamma)
    \gamma^{-\frac{\gamma}{1+\gamma}} \Bigl(
      (2 \varepsilon(\delta))^6\wedge1
      \Bigr)^{-\frac1{1+\gamma}}
      |\eta-y|^{\frac2{1+\gamma}}.
\]
Thus, by \eqref{l.upper.21}, letting
$\delta\nearrow\delta_0$, we obtain \eqref{l.lower.1}.
\end{proof}

\begin{remark}
Dominating $\boldsymbol{v}_{0,0}$ by $M_{1,0}^{2\gamma}$ is
best in the sense that
\[
    \sup_{w\in \mathcal{W}_1^0} 
      \frac{\boldsymbol{v}_{0,0}(w)}{M_{1,0}(w)^{2\gamma}}
    =1.
\]
This identity is seen as follows.
It is obvious that the supremum is bounded by $1$ from
above. 
For $n\in \mathbb{N}$, take 
$w_n=(w_n^1,\dots,w_n^d)\in \mathcal{W}_1^0$ given by 
\[
    w_n^1(t)
     =\begin{cases}
        nt & \text{for }t\in[0,\frac1{n}),
        \\
        1  & \text{for }t\in[\frac1{n},1-\frac1{n}),
        \\
        1+n\Bigl(t-\Bigl(1-\dfrac1{n}\Bigr)\Bigr)       
        & \text{for }t\in[1-\frac1{n},1],
      \end{cases}
     \quad
     w_n^2=\dots=w_n^d=0.
\]
Then
\[
    \frac{\boldsymbol{v}_{0,0}(w_n)}{M_{1,0}(w_n)^{2\gamma}}
    \ge 1-\frac2{n}
    \quad\text{for }n\in \mathbb{N}.
\]
Letting $n\to\infty$, we obtain that the supremum is
bounded by $1$ from below.
\end{remark}

\section{Remark}

For $\gamma\ge1$, we can sharpen the assertion of 
Lemma~\ref{lem.y=eta}.

Define 
\[
    f_{x,\xi}(a)=\int_{\mathcal{W}_1^0} \biggl(
     \int_0^1|a \beta(t)+\ell^{1,x,\xi}(t)|^{2\gamma}dt
     \biggr)^{-\frac{d^\prime}2} d\mu_1^0
    \quad\text{for }a\in \mathbb{R}.
\]
By \eqref{l.p_T.alt.1}, it holds 
\[
    p_T((x,y),(\xi,y))
    =\frac1{\sqrt{2\pi T}^{d+d^\prime}}
     \exp\Bigl(-\frac{|\xi-x|^2}{2T}\Bigr)
     f_{x,\xi}(\sqrt{T})
    \quad\text{for $T>0$ and $y\in \mathbb{R}^{d^\prime}$}.
\]

Modifying the proof of Lemma~\ref{lem.non-deg} slightly, we
have 
\[
    \sup_{|a|\le R} \int_{\mathcal{W}_1^0}
      \biggl(\int_0^1
        |a \beta(t)+\ell^{1,x,\xi}(t)|^{2\gamma} dt
      \biggr)^{-p} d\mu_1^0<\infty
     \quad\text{for any $p\in(1,\infty)$ and $R>0$}.
\]
This implies that $f_{x,\xi}\in C^{m(\gamma)}(\mathbb{R})$,
where
\[
    m(\gamma)
    =\begin{cases} 
      \infty & \text{if }\gamma\in \mathbb{N},
      \\
      [\gamma] & \text{if }\gamma\notin \mathbb{N}.
     \end{cases}
\]

By the Taylor expansion, we have
\[
    f_{x,\xi}(a)=f_{x,\xi}(0)+\sum_{k=1}^m
    \frac{f_{x,\xi}^{(k)}(0)}{k!} a^k +O(a^{m+1})
    \quad\text{for any $m<m(\gamma)$ as $a\to0$}.
\]
Since 
$\{\beta(t)\}_{t\in[0,1]}\overset{\text{law}}{\sim}
 \{-\beta(t)\}_{t\in[0,1]}$, it holds
\[
    f_{x,\xi}(-a)=f_{x,\xi}(a)
    \quad\text{for }a\in \mathbb{R}.
\]
This yields
$(-1)^k f_{x,\xi}^{(k)}(-a)=f_{x,\xi}^{(k)}(a)$
for $k\in \mathbb{N}$ and $a\in \mathbb{R}$.
In particular, 
\[
    f_{x,\xi}^{(k)}(0)=0
    \quad\text{for }k\notin 2 \mathbb{N}.
\]
Hence the Taylor expansion  comes down to
\[
        f_{x,\xi}(a)=f_{x,\xi}(0)
         +\sum_{k\in \mathbb{N},2k\le m}
    \frac{f_{x,\xi}^{(2k)}(0)}{(2k)!} a^{2k} +O(a^{m+1})
    \quad\text{for any $m<m(\gamma)$ as $a\to0$}.
\]
Thus we obtain
\begin{align*}
   p_T((x,y),(\xi,y))
   =& \frac1{\sqrt{2\pi T}^{d+d^\prime}}
    \exp\Bigl(-\frac{|\xi-x|^2}{2T}\Bigr)
   \\
    & \qquad\times
    \biggl\{
      \biggl(\int_0^1|\ell^{1,x,\xi}(t)|^{2\gamma} dt
        \biggr)^{-\frac{d^\prime}2}
      +\sum_{k\in \mathbb{N},2k\le m}
         \frac{f_{x,\xi}^{(2k)}(0)}{(2k)!} T^k 
      +O\Bigl(T^{\frac{m+1}2}\Bigr)
    \biggr\}
\end{align*}
as $T\to 0$ for any $m<m(\gamma)$.

\medskip\noindent

{\bf Acknowledgments}.
This work was supported by JSPS KAKENHI Grant Number
JP18K03336.


\bigskip
\hfill
\begin{minipage}{170pt}
Setsuo Taniguchi
\\
Faculty of Arts and Science 
\\
Kyushu University
\\
Fukuoka 819-0395, Japan
\\
(se2otngc@artsci.kyushu-u.ac.jp)
\end{minipage}


\begin{thebibliography}{9}

\bibitem{baouendi}
M. Baouendi,
Sur une classe d'op\'erateurs elliptiques deg\'en\'er\'es,
Bull. Soc. Math. France, {\bf 95} (1967), 45--87.

\bibitem{bfi}
W. Bauer, K. Furutani, and C. Iwasaki,
Fundamental solution of a higher step Grushin type operator,
Adv. Math., {\bf 271} (2015), 188--234.

\bibitem{ccfi}
O. Calin, D.-C. Chang, K. Furutani, C. Iwasaki,
Heat Kernels for Elliptic and Sub-elliptic Operators,
Birkh\"{a}user, 2011

\bibitem{ccggl}
C.-H. Chang, D.-C. Chang, B. Gaveau, P. Greiner, and
H.-P. Lee,
Geometric analysis on a step 2 Grushin operator,
Bull. Inst. Math. Acad. Sinica (New Ser.), {\bf 4} (2009),
119--188.



\bibitem{dembo_zeitouni}
A. Dembo and O. Zeitouni,
Large deviations techniques and applications, 2nd Ed.
Springer, 1998.

\bibitem{grushin1}
V. Grushin,
On a class of hypoelliptic operators, 
Math. USSR Sbornik, {\bf 12} (1970), 458--476.

\bibitem{grushin2}
V. Grushin,
On a class of hypoelliptic operators degenerate on a
submanifold,
Math. USSR Sbornik, {\bf 13} (1971), 155--186.

\bibitem{ledoux}
M. Ledoux,
Isoperimetry and Gaussian analysis, 
in ``Lectures on Probability Theory and Statistics: Ecole
d'Ete de Probabilites de St. Flour XXIV - 1994'',
Lecture Notes in Math. {\bf 1648}, Springer, New York,
1996, pp. 165--294.

\bibitem{shorack_wellner}
G. Shorack and J. Wellner,
Empirical processes eith applications to statistics,
John Wiley \& Sons, New York, 1986.

\bibitem{stroock}
D.W. Stroock,
Probability Theory, An Analytic View, 2nd Ed.
Cambridge Univ. Press, 2010.

\bibitem{st-grushin}
S. Taniguchi,
An application of the partial Malliavin calculus to
Baouendi-Grushin operators,
Kyushu J. Math., {\bf 73} (2019), 417--431.

\bibitem{triebel}
H. Triebel,
Interpolation Theory, Function Spaces, Differential Operators,
North-Holland, Amsterdam, 1978.

\end{thebibliography}
\end{document}